\documentclass[12pt]{article}

\usepackage{amsmath,amssymb,amsthm,accents}
\usepackage{amscd}
\usepackage[initials,short-journals,non-sorted-cites]{amsrefs}
\usepackage{enumitem}
\usepackage{type1cm}
\usepackage{float}
\usepackage[dvipdfmx]{graphicx}
\usepackage[stable]{footmisc}

\usepackage{color}

\newtheorem{theo}{Theorem}[section]
\newtheorem{prop}[theo]{Proposition}
\newtheorem{lemm}[theo]{Lemma}
\newtheorem{cor}[theo]{Corollary}

\theoremstyle{definition}
\newtheorem{defi}[theo]{Definition}
\newtheorem{exam}[theo]{Example}

\newtheorem{remark}[theo]{Remark}

\newcommand{\Z}{\mathbb{Z}}

\newcommand{\R}{\mathbb{R}}

\title{Quandle coloring quivers of links using dihedral quandles\footnotemark}

\author{Yuta Taniguchi}
\date{}

\begin{document}

\maketitle
\footnotetext[\ast]{
{\bf keywords:} {quandle, quandle coloring quiver, dihedral quandle}
}

\section{Introduction}

 A {\it quandle} (\cite{Joy, Mat}) is an algebraic structure defined on a set with a binary operation whose definition was motivated from knot theory. D.~Joyce and S.~V.~Matveev \cite{Joy, Mat} associated a quandle to a link, which is so-called the {\it link quandle}  or the {\it fundamental quandle} of a link. 
 Since then many link invariants using link quandles have been introduced and studied. 
  A typical and elementary example is the {\it quandle coloring number} which is the cardinal number of the set of quandle homomorphisms from the link quandle to a fixed finite quandle. A {\it quandle cocycle invariant} \cite{CJKLS} and a {\it shadow quandle cocycle invariant} 
 (cf. \cite{CKS2004})
  are also such link invariants, which are enhancements of the quandle coloring number.

 In $2019$, K. Cho and S. Nelson \cite{Cho2019quiver} introduced the notion of a {\it quandle coloring quiver}, which is a quiver-valued
link invariant, and gave interesting examples.  
 This invariant is defined when we fix a finite quandle and a set of its endomorphisms. 
They also introduced in \cite{Cho2019cocycle} the notion of a {\it quandle cocycle quiver} which is an enhancement of the quandle coloring quiver by assigning to each vertex  a weight computed using a quandle $2$-cocycle. 
 
In this paper, we study quandle coloring quivers using dihedral quandles. 
We show that, when we use a dihedral quandle of prime order,  
the quandle coloring quivers are equivalent to the quandle coloring numbers
 (Theorem \ref{p-dihed}). The notions of a quandle coloring quiver and a quandle cocyle quiver are naturally generalized to a {\it shadow quandle coloring quiver} and a {\it shadow quandle cocycle quiver}.  
We show that, when we use a dihedral quandle of prime order and {\it Mochizuki's $3$-cocycle}, the shadow quandle cocycle quivers are equivalent to  the shadow quandle cocycle invariants (Theorem \ref{cocycle-quiver}).
  
This paper is organaized as follows.
 In Section \ref{Quandles}, we recall the definition of a quandle, a quandle coloring and a shadow quandle cocycle invariant.
 In Section \ref{quivers-prime}, we recall the definition of the quandle coloring quiver of a link, and discuss quandle coloring quivers using a  dihedral quandle of prime order.
In Section \ref{quivers-composite}, we discuss quandle coloring quivers of a  dihedral quandle of composite order.
In Section \ref{shadow quiver}, we introduce the notion of a shadow quandle coloring quiver and a shadow quandle cocycle quiver. 
In Section \ref{shadow quiver-Mochizuki}, we discuss shadow quandle cocycle quivers using a dihedral quandle of prime order and Mochizuki's $3$-cocycle.

\section{Quandles and quandle cocycle invariants}
\label{Quandles}

\subsection{Quandles}
A {\it quandle} is a set $X$ with a binary operation $\ast:X\times X\to X$ satisfying the following three axioms.
\vspace{-0.5\baselineskip}
\begin{itemize}
	\setlength{\itemsep}{0pt}
	\setlength{\parskip}{0pt}
\item[(Q$1$)] For any $x\in X$, we have $x\ast x=x$.
\item[(Q$2$)] For any $y\in X$, the map $\ast y:X\to X$, $x\mapsto x\ast y$ is a bijection.
\item[(Q$3$)] For any $x,y,z\in X$, we have $(x\ast y)\ast z=(x\ast z)\ast(y\ast z)$.
\end{itemize}
\vspace{-0.5\baselineskip}
These axioms correspond to the three kinds of Reidemeister moves (cf. \cite{Joy, Mat}). 

\begin{exam}
The {\it dihedral quandle} of order $n$, denoted by $R_n$, is $\Z_n = \Z / n \Z$ with an operation $\ast$ defined by $x\ast y=2y-x$. 
\end{exam}

\begin{exam}
Let $M$ be a $\Z[t^{\pm 1}]$-module. We define an operation by $x\ast y=tx+(1-t)y$. Then, $M$ is a quandle, which is called an {\it Alexander quandle}.
\end{exam}

A map $f:X\to Y$ between quandles is called a 
 ({\it quandle}) {\it homomorphism} if $f(x\ast y)=f(x)\ast f(y)$ for any $x,y\in X$. 
A quandle homomorphism  $f:X\to Y$ is 
a {\it quandle isomorphism}, a {\it quandle endomorphism}, or a 
{\it quandle automorphism} if it is a bijection, if $X=Y$, or if $X=Y$ and it is a bijection, respectively.  
We denote by Hom$(X,Y)$ the set of quandle homomorphisms from $X$ to $Y$, by End$(X)$ the set of quandle endomorphisms of $X$ and by Aut$(X)$ the set of quandle automorphisms of $X$. 

\subsection{Quandle colorings and shadow cocycle invariants}
Let $X$ be a quandle and  $D$ be an oriented link diagram on $\R^{2}$. We denote the set of the arcs of $D$ by Arc$(D)$. A map $c:{\rm Arc}(D)\to X$ is an $X$-{\it coloring} if $c$ satisfies following condition at every crossing of $D$.
\vspace{-0.5\baselineskip}
\begin{itemize}
	\setlength{\itemsep}{0pt}
	\setlength{\parskip}{0pt}
\item Let $x_i,x_j,x_k$ be arcs around a crossing as in Figure \ref{coloringcondition} (left). Then, $c(x_i)\ast c(x_j)=c(x_k)$.
\end{itemize}
\vspace{-0.5\baselineskip}
The value  $c(a)$ assigned to arc $a$ is called the {\it label}. An $X$-coloring $c$ is a {\it trivial} coloring if c is a constant map.  We denote by ${\rm Col}_{X}(D)$ the set of $X$-colorings of $D$. 

 \begin{figure}[H]
  	\begin{center}
  	\includegraphics[height=3cm,width=8cm]{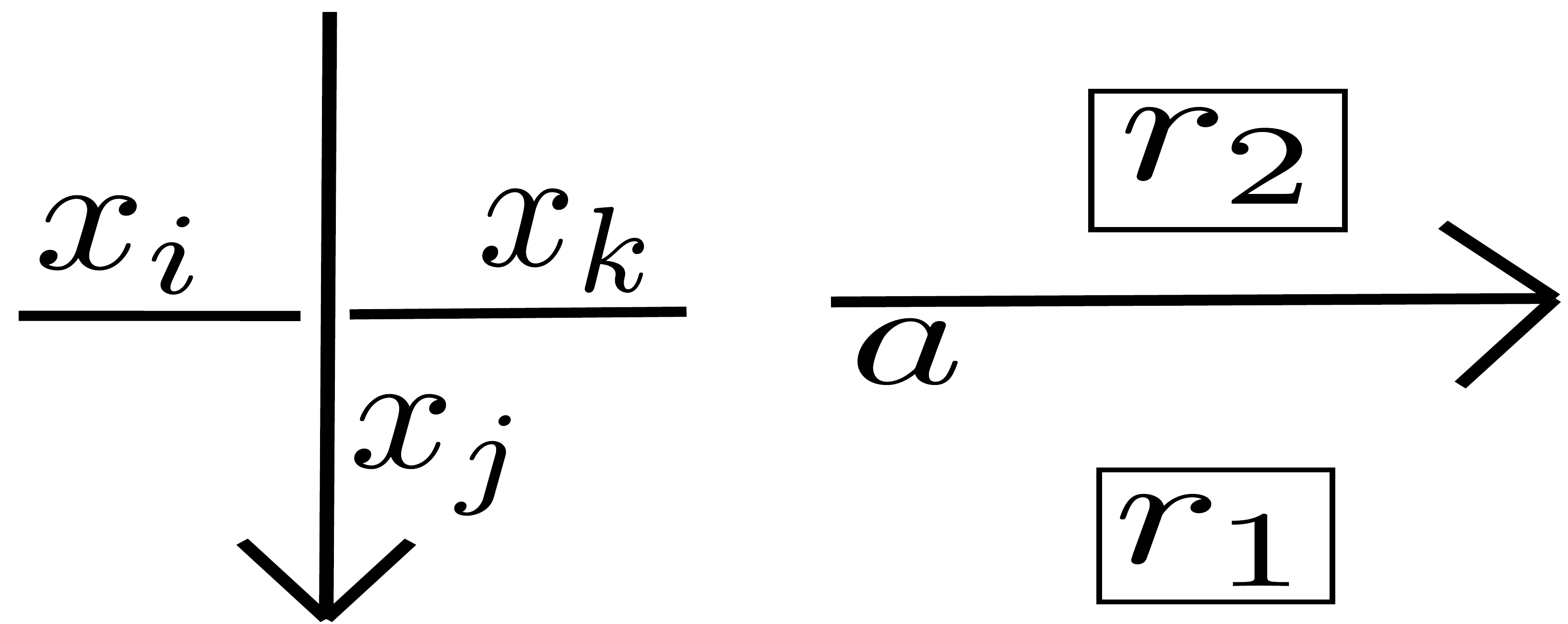}
  	\caption{coloring condition.}
  	\label{coloringcondition}
  	\end{center}
  \end{figure}

If two link diagrams $D$ and $D^{\prime}$ are related by Reidemeister moves, we have a bijection between Col$_{X}(D)$ and Col$_{X}(D^{\prime})$. Thus, when $X$ is a finite quandle, the cardinal number $| {\rm Col}_{X}(D)|$ of ${\rm Col}_{X}(D)$
 is a link invariant. It is called the $X$-{\it coloring number} or the {\it coloring number} by $X$.
 
 Let $|D|$ be immersed circles in $\R^{2}$ obtained from $D$ by ignoring over/under information of the crossings. We denote by ${\rm Region}(D)$ the set of connected component of $\R^{2}\backslash |D|$ and by $r_{\infty}$ the unbounded region of Region$(D)$. A map $c_{\ast}:{\rm Arc}(D)\cup{\rm Region}(D)\to X$ is a {\it shadow $X$-coloring} if $c_{\ast}$ satisfies the following conditions.
 \vspace{-0.5\baselineskip}
\begin{itemize}
	\setlength{\itemsep}{0pt}
	\setlength{\parskip}{0pt}
\item $c_{\ast}|_{{\rm Arc}(D)}$ is an $X$-coloring.
\item Let $r_1$ and $r_2$ be adjacent regions of $D$ along an arc $a$ as in Figure \ref{coloringcondition} (right). Then, $c_{\ast}(r_1)\ast c_{\ast}(a)=c_{\ast}(r_2)$.
\end{itemize}
\vspace{-0.5\baselineskip}
We denote by ${\rm SCol}_{X}(D)$ the set of shadow $X$-colorings of $D$ and by ${\rm SCol}_{X}(D,a)$ the set of shadow $X$-colorings of $D$ which satisfy $c_{\ast}(r_{\infty})=a$ for $a\in X$. 

Let $A$ be an abelian group. A map $\theta:X^3\to A$ is called a {\it quandle $3$-cocycle} if $\theta$ satisfies the following conditions (cf. \cite{CJKLS}).
\vspace{-0.5\baselineskip}
\begin{itemize}
	\setlength{\itemsep}{0pt}
	\setlength{\parskip}{0pt}
\item For any $x,y\in X$, $\theta(x,x,y)=\theta(x,y,y)=0$.
\item For any $x,y,z,w\in X$, $\theta(x,y,z)+\theta(x\ast z,y\ast z,w)+\theta(x,z,w)=\theta(x\ast y,z,w)+\theta(x,y,w)+\theta(x\ast w,y\ast w,z\ast w)$.
\end{itemize}
\vspace{-0.5\baselineskip}

\begin{exam}(\cite{Moc2003, Moc2011})
Let $p$ be an odd prime. We define a map $\theta_p:(R_p)^3\to\Z_p$ by
\[
\theta_p(x,y,z)=(x-y)\frac{(y^{p}+(2y-z)^p-2z^p)}{p}.
\]
Then, $\theta_p$ is a quandle $3$-cocycle, which is referred to as {\it Mochizuki's $3$-cocycle} in this paper.
\end{exam}

Let $\theta:X^3\to A$ be a quandle $3$-cocycle of a quandle $X$ and  $D$ be an oriented link diagram. For a shadow $X$-coloring $c_{\ast}$ of $D$, we associated a weight $\pm\theta(x,y,z)$ to each crossing of $D$, where $x \in X$ is the label of the region and $y, z \in X$ are labels of the arcs indicated in Figure \ref{cocyclecondition}. 
Then, we sum up the weights all over the crossings of $D$ to obtain an element of $A$ denoted by $\Phi_{\theta}(D,c_{\ast})$. 

 \begin{figure}[H]
  	\begin{center}
  	\includegraphics[height=3.5cm,width=8cm]{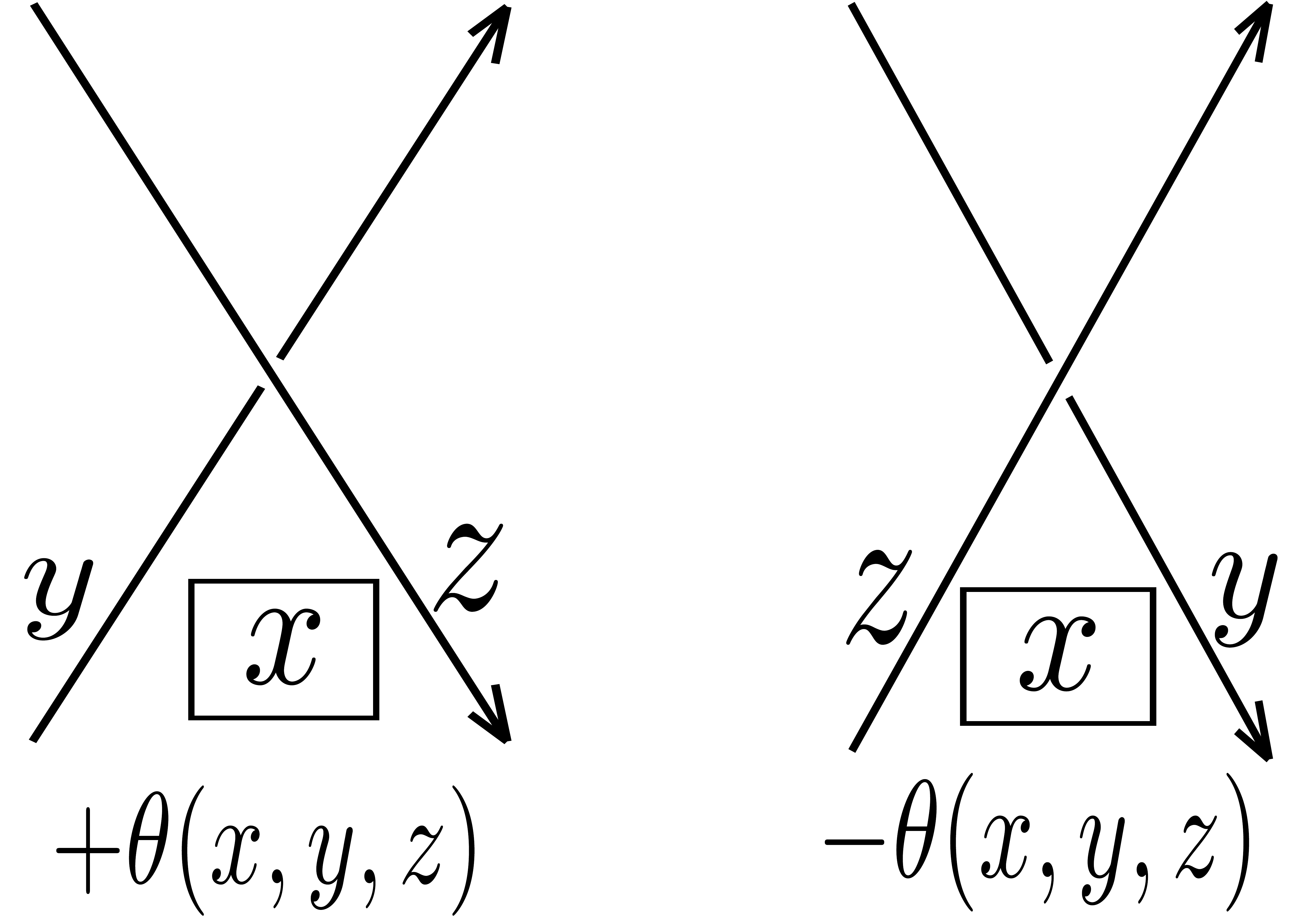}
  	\caption{The weight of a crossing.}
  	\label{cocyclecondition}
  	\end{center}
  \end{figure}
  
When $X$ is a finite quandle, $\Phi_{\theta}(D)=\{\Phi_{\theta}(D,c_{\ast})\mid c_{\ast}\in{\rm SCol}_{X}(D)\}$ as a multiset is a link invariant, which we call a {\it shadow quandle cocycle invariant}  (cf. \cite{CKS}).

\section{Quandle coloring quivers using a dihedral quandle of prime order}
\label{quivers-prime}

Let $X$ be a finite quandle and $D$ be an oriented link diagram. For any subset $S\subset{\rm End}(X)$, the {\it quandle coloring quiver} of $D$, which is denoted by $Q^{S}_{X}(D)$, is the quiver with a vertex for each $X$-coloring $c \in{\rm Col}_{X}(D)$ and an edge directed from $v$ to $w$ when $w=f\circ v$ for an element $f \in S$. 
 In \cite{Cho2019quiver}, it is proved that $Q^{S}_{X}(D)$ is a link invariant. Precisely speaking, if two link diagrams $D$ and $D^{\prime}$ are related by Reidemeister moves, then the  quandle coloring quivers $Q^{S}_{X}(D)$ and $Q^{S}_{X}(D^{\prime})$ are isomorphic as quivers for any $S\subset$ End$(X)$.

\begin{remark}
Note that the quandle coloring number of $D$ by using a finite quandle $X$ is $| {\rm Col}_{X}(D) |$, which is the number of vertices of $Q^{S}_{X}(D)$.  Thus, the quandle coloring quiver $Q^{S}_{X}(D)$ is in general a stronger link invariant than the quandle coloring number. For example, both of the knots $8_{10}$ and $8_{18}$ have the same quandle coloring number by using the dihedral quandle $R_9$ of order $9$, which is $81$.  On the other hand, their quandle coloring quivers using $R_9$ with $S= {\rm End}(R_9)$ are not isomorphic as quivers (Example $11$ of \cite{Cho2019quiver}).  
\end{remark}

\begin{lemm}
\label{alpha}
Let $X$ be a finite quandle and let $D$ and $D^{\prime}$ be oriented link diagrams. Assume that there exists a bijection $\varphi:{\rm Col}_{X}(D)\to{\rm Col}_{X}(D^{\prime})$ which satisfies the following condition, referred to as the condition $(\alpha)$, 
\begin{itemize}
\item For any $f\in{\rm End}(X)$ and $c\in{\rm Col}_{X}(D)$, 
$f\circ\varphi(c)=\varphi(f\circ c)$. 
\end{itemize}
Then the quandle coloring quivers $Q^{S}_{X}(D)$ and $Q^{S}_{X}(D^{\prime})$ are isomorphic for any $S\subset{\rm End}(X)$. 
\end{lemm}

\begin{proof}
If $(v,w)$ is an edge of $Q^{S}_{X}(D)$, we have
\begin{eqnarray*}
\varphi(w)&=&\varphi(f\circ v)\\
&=&f\circ\varphi(v)
\end{eqnarray*}
for some $f\in S$. Then, $(\varphi(v),\varphi(w))$ is an edge of $Q^{S}_{X}(D^{\prime})$. As the same way, it holds that $(v,w)$ is an edge of $Q^{S}_{X}(D)$ if and only if $(\varphi(v),\varphi(w))$ is an edge of $Q^{S}_{X}(D^{\prime})$. Therefore, $\varphi$ is an isomorphism between $Q^{S}_{X}(D)$ and $Q^{S}_{X}(D^{\prime})$.
\end{proof}

The following result is one of the main results of this paper.

\begin{theo}
\label{p-dihed}
Let $D$ and $D^{\prime}$ be oriented link diagrams and  $p$ be a prime. 
For any $S\subset{\rm End}(R_{p})$, the quandle coloring quivers $Q^{S}_{R_p}(D)$ and $Q^{S}_{R_p}(D^{\prime})$ are isomorphic if and only if 
$|{\rm Col}_{R_p}(D)|=|{\rm Col}_{R_p}(D^{\prime})|$. 
\end{theo}

\begin{lemm}
\label{end-dihed}
Let $n$ be a positive integer greater than $1$. If $f:R_n\to R_n$ is a quandle homomorphism, then there exist unique elements $a,b\in \Z_n$ such that $f(x)=ax+b$ for any $x\in X$. 
\end{lemm}

\begin{proof}
Put $a=f(1)-f(0)$ and $b=f(0)$. We define a map $g:R_n\to R_n$ by $g(x)=ax+b$. It is seen by a directed calculation that $g$ is a quandle homomorphism. By definition of $g$, we have $f(0)=g(0)$ and $f(1)=g(1)$. Recall that $(m-1)\ast m=m+1$ for any $m\in R_n$, and we have $g(m)=f(m)$ for any $m\in R_n$. The uniqueness of $a$ and $b$ is seen by evaluating $f$ with $0$ and $1$.  
\end{proof}

\begin{remark}
\label{endo-auto}
If $f$ is a quandle automorphism of $R_n$, then $a$ is an element of $\Z^{\times}_n$. When $p$ is a prime, we have ${\rm End}(R_p)={\rm Aut}(R_p)\cup\{{\rm constant\ maps}\}$.
\end{remark}

\begin{proof}[Proof of Theorem \ref{p-dihed}]
The only if part is trivial. We consider the if part.

We show that there exists a bijection $\varphi:{\rm Col}_{R_p}(D)\to{\rm Col}_{R_p}(D^{\prime})$ which satisfies the condition $(\alpha)$ in Lemma~\ref{alpha}.

Recall that ${\rm Col}_{R_p}(D)$ and ${\rm Col}_{R_p}(D^{\prime})$ are 
$\Z_p$-vector spaces. By assumption, they have the same cardinality and hence their dimensions are the same, say $n$.  Take a basis $\{c_1,\dots,c_n\}$ of Col$_{R_p}(D)$ and a basis $\{c^{\prime}_1,\dots,c^{\prime}_n\}$ of Col$_{R_p}(D^{\prime})$ such that $c_1$ and $c^{\prime}_1$ are the trivial coloring which is the constant map onto $1$. 
Let $f$ be a quandle endomorphism of $R_p$. By Lemma~\ref{end-dihed}, we have $f(x)=ax+b$ for some $a,b\in\Z_p$. Then, 
\vspace{-0.5\baselineskip}
\begin{eqnarray*}
f\circ c(x)&=&f(c(x))\\
&=&ac(x)+b\\
&=&ac(x)+bc_1(x)
\end{eqnarray*}
for any $c\in{\rm Col}_{R_p}(D)$ and $x\in{\rm Arc}(D)$. Therefore, we have $f\circ c=ac+bc_1$ for any $c\in{\rm Col}_{R_p}(D)$. Similarly, we can show that $f\circ c^{\prime}=ac^{\prime}+bc^{\prime}_1$ for any $c^{\prime}\in{\rm Col}_{R_p}(D^{\prime})$.

Since ${\rm Col}_{R_p}(D)$ and ${\rm Col}_{R_p}(D^{\prime})$ are 
$\Z_p$-vector spaces of the same dimension $n$, there is a $\Z_p$-linear isomorphism  $\varphi:{\rm Col}_{R_p}(D)\to{\rm Col}_{R_p}(D^{\prime})$ such that $\varphi(c_1)=c^{\prime}_1$. Then,
\vspace{-0.5\baselineskip}
\begin{eqnarray*}
\varphi(f\circ c)&=&\varphi(ac+bc_1)\\
&=&a\varphi(c)+b\varphi(c_1)\\
&=&a\varphi(c)+bc^{\prime}_1\\
&=&f\circ\varphi(c).
\end{eqnarray*} 
Therefore, $\varphi$ satisfies the condition $(\alpha)$.
\end{proof}

\section{Quandle coloring quivers using a dihedral quandle of composite order}
\label{quivers-composite}

We discuss quandle coloring quivers using a dihedral quandle of composite order. 

\begin{theo}
\label{coprime-dihed}
Let $D$ and $D^{\prime}$ be oriented link diagrams and let $m$ and $n$ be coprime integers greater than $1$. Suppose that there exists bijections $\varphi_m:{\rm Col}_{R_m}(D)\to{\rm Col}_{R_m}(D^{\prime})$ and $\varphi_n:{\rm Col}_{R_n}(D)\to{\rm Col}_{R_n}(D^{\prime})$ which satisfy the condition $(\alpha)$. Then $Q^{S}_{R_{mn}}(D)$ and $Q^{S}_{R_{mn}}(D^{\prime})$ are isomorphic for any $S\subset{\rm End}(R_{mn})$.
\end{theo}

\begin{proof}
Since $m$ and $n$ are coprime, by Chinese remainder theorem, we have a bijection $R_{mn}\to R_m\times R_n$ by $[x]_{mn}\mapsto ([x]_m,[x]_n)$, where $[x]_{m}$ means an element of $R_m = \Z / m \Z$ represented by an integer $x$. It is not only a group isomorphism but also a quandle isomorphism.  Thus $R_{mn}$ can be identified with $R_m\times R_n$ as a quandle.  
On the other hand, we can naturally idenitfy ${\rm Col}_{R_{m}\times R_{n}}(D)$ with ${\rm Col}_{R_{m}}(D)\times{\rm Col}_{R_{n}}(D)$. Then, we have a natural bijection 
$$\psi:{\rm Col}_{R_{mn}}(D)\to{\rm Col}_{R_m}(D)\times{\rm Col}_{R_n}(D).$$  
And similarly we have a natural bijection 
$\psi^{\prime}:{\rm Col}_{R_{mn}}(D^{\prime})\to{\rm Col}_{R_m}(D^{\prime})\times{\rm Col}_{R_n}(D^{\prime})$.  

Let us define a map 
$$\psi_m:{\rm End}(R_{mn})\to{\rm End}(R_m)$$ 
by $\psi_m(f)([x]_m)=[a]_m[x]_m+[b]_m$ for any $ f \in {\rm End}(R_{mn})$, where 
$[a]_m, [b]_m \in R_m$ are images of  $[a]_{mn}, [b]_{mn} \in R_{mn}$ which are uniquely 
determined from $f$ by Lemma~\ref{end-dihed} with  
$f([x]_{mn})=[a]_{mn}[x]_{mn}+[b]_{mn}$ for $[x]_{mn} \in R_{mn}$. 
Similarly, we define a map $\psi_{n}:{\rm End}(R_{mn})\to{\rm End}(R_n)$. 
 By a direct calculation, we see that $\psi(f\circ c)=(\psi_{m}(f)\times\psi_{n}(f))\circ(\psi(c))$ and 
$\psi^{\prime}(f\circ c^{\prime})=(\psi_{m}(f)\times\psi_{n}(f))\circ(\psi(c^{\prime}))$ 
for any $f \in {\rm End}(R_{mn})$, $c \in {\rm Col}_{R_{mn}}(D)$ and $c^{\prime} \in {\rm Col}_{R_{mn}}(D^{\prime})$.

By assumption, there exists bijections $\varphi_m:{\rm Col}_{R_m}(D)\to{\rm Col}_{R_m}(D^{\prime})$ and $\varphi_n:{\rm Col}_{R_n}(D)\to{\rm Col}_{R_n}(D^{\prime})$  which satisfy the condition $(\alpha)$. 
Let $\varphi=\varphi_m\times\varphi_n:{\rm Col}_{R_{m}}(D)\times{\rm Col}_{R_{n}}(D)\to{\rm Col}_{R_{m}}(D^{\prime})\times{\rm Col}_{R_{n}}(D^{\prime})$ and define a bijection $\Psi$ by $\Psi=(\psi^{\prime})^{-1}\circ\varphi\circ\psi:{\rm Col}_{R_{mn}}(D)\to{\rm Col}_{R_{mn}}(D^{\prime})$. Then, we have
 \begin{eqnarray*}
 \Psi(f\circ c)&=&(\psi^{\prime})^{-1}\circ\varphi((\psi_{m}(f)\times\psi_{n}(f))\circ(\psi(c)))\\
 &=&(\psi^{\prime})^{-1}((\psi_{m}(f)\times\psi_{n}(f))\circ(\varphi\circ\psi(c)))\\
 &=&f\circ\Psi(c)
 \end{eqnarray*}
 for any $f\in{\rm End}(R_{mn})$ and $c\in{\rm Col}_{R_{mn}}(D)$. Therefore, $\Psi$ satisfies the condition $(\alpha)$, which implies by Lemma~\ref{alpha} the assertion.  
\end{proof}

\begin{cor}
\label{coprime-cor}
Let $D$ and $D^{\prime}$ be oriented link diagrams and $P=p^{e_1}_1p^{e_2}_2\dots p^{e_n}_n$ be the prime factorization of a positive integer $P$. If there exists a bijection $\varphi_{p^{e_i}_i}:{\rm Col}_{R_{p^{e_i}_i}}(D)\to{\rm Col}_{R_{p^{e_i}_i}}(D^{\prime})$ which satisfies the condition $(\alpha)$ for each $i$, then  $Q^{S}_{R_{P}}(D)$ and $Q^{S}_{R_{P}}(D^{\prime})$ are isomorphic for any $S\subset{\rm End}(R_{P})$. 
\end{cor}

\begin{proof}
By assumption, there exists bijections $\varphi_{p^{e_1}_1}:{\rm Col}_{R_{p^{e_1}_1}}(D)\to{\rm Col}_{R_{p^{e_1}_1}}(D^{\prime})$ and $\varphi_{p^{e_2}_2}:{\rm Col}_{R_{p^{e_2}_2}}(D)\to{\rm Col}_{R_{p^{e_2}_2}}(D^{\prime})$. As seen in the proof of Theorem \ref{coprime-dihed}, we have a bijection $\varphi_{p^{e_1}_1p^{e_2}_2}:{\rm Col}_{R_{p^{e_1}_1p^{e_2}_2}}(D)\to{\rm Col}_{R_{p^{e_1}_1p^{e_2}_2}}(D^{\prime})$ which satisfies the condition $(\alpha)$. Repeating this procedure, we have a bijection $\varphi_{P}:{\rm Col}_{R_P}(D)\to{\rm Col}_{R_P}(D^{\prime})$ which satisfies the condition $(\alpha)$, which implies by Lemma\ref{alpha} the assertion.
\end{proof}

\begin{theo}
\label{cor-dihed}
Let $D$ and $D^{\prime}$ be oriented link diagrams and  $P=p_1p_2\dots p_n$ be a positive integer for some distinct primes $p_1,\dots,p_n$. For any $S\subset{\rm End}(R_P)$, the quandle coloring quivers $Q^{S}_{R_P}(D)$ and $Q^{S}_{R_P}(D^{\prime})$ are isomorphic if and only if $|{\rm Col}_{R_P}(D)|=|{\rm Col}_{R_P}(D^{\prime})|$. 
\end{theo}

\begin{proof}
The only if part is trivial. We show the if part. Suppose that $|{\rm Col}_{R_P}(D)|=|{\rm Col}_{R_P}(D^{\prime})|$.
As seen in the proof of Theorem \ref{coprime-dihed}, ${\rm Col}_{R_P}(D)$ is naturally identified with 
$\Pi^{n}_{i=1}{\rm Col}_{R_{p_i}}(D)$.  Since $|{\rm Col}_{R_P}(D)|=|{\rm Col}_{R_P}(D^{\prime})|$, we have 
 $|{\rm Col}_{R_{p_i}}(D)|=|{\rm Col}_{R_{p_i}}(D^{\prime})|$ for each $i$. 
 As seen in the proof of Theorem \ref{p-dihed}, we have a bijection $\varphi_{p_i}:{\rm Col}_{R_{p_i}}(D)\to{\rm Col}_{R_{p_i}}(D^{\prime})$ which satisfies the condition $(\alpha)$ for each $i$. By Corollary \ref{coprime-cor}, we have a bijection $\varphi_P:{\rm Col}_{R_P}(D)\to{\rm Col}_{R_P}(D^{\prime})$ which satisfies the condition $(\alpha)$, which implies by Lemma~\ref{alpha}
 the assertion.
\end{proof}

\section{Shadow versions of quandle quivers}
\label{shadow quiver}

In this section we define the {\it shadow quandle coloring quiver} and the {\it shadow quandle cocycle quiver}, which are shadow versions of the quandle coloring quiver and the quandle cocycle quiver defined by Cho and Nelson \cite{Cho2019quiver, Cho2019cocycle}.

\begin{defi}
Fix a finite quandle $X$, an element $a\in X$ and a subset $S\subset{\rm End}(X)$.    
The {\it shadow quandle coloring quiver} of an oriented link diagram $D$, which is denoted by $SQ^{S}_{X}(D,a)$, is the quiver with a vertex for each shadow $X$-coloring $c_{\ast}\in{\rm SCol}_{X}(D,a)$ and an edge directed from $v_{\ast}$ to $w_{\ast}$ when $w_{\ast}|_{{\rm Arc}(D)}=f\circ v_{\ast}|_{{\rm Arc}(D)}$ for an element $f\in S$.
\end{defi}

The shadow quandle coloring quiver is a link invariant.  However, it is nothing more than the quandle coloring quiver as seen below.

\begin{prop}
\label{shadow-quiver}
The quandle coloring quiver $Q^{S}_{X}(D)$ and the shadow quandle coloring quiver $SQ^{S}_{X}(D,a)$ are isomorphic for any $S\subset{\rm End}(X)$ and $a\in X$.
\end{prop}
\begin{proof}
For any $a\in X$ and $X$-coloring $c\in{\rm Col}_{X}(D)$, there is a unique shadow $X$-coloring $c_{\ast}$ such that $c_{\ast}|_{{\rm Arc}(D)}=c$ and $c_{\ast}(r_{\infty})=a$ (cf. \cite{CKS}). Hence, we have a natural bijection $\varphi:{\rm Col}_{X}(D)\to{\rm SCol}_{X}(D,a)$. By definition of the shadow quandle coloring quiver, we see that $\varphi$ is a quiver isomorphism between $Q^{S}_{X}(D)$ and $SQ^{S}_{X}(D,a)$ for any subset $S\subset{\rm End}(X)$.
\end{proof}

Next, we define the shadow quandle cocycle quiver.

\begin{defi}
Fix a finite quandle $X$, an element $a\in X$, a subset $S\subset{\rm End}(X)$,  
an abelian group $A$ and a quandle $3$-cocycle $\theta:X^3\to A$.  
The {\it shadow quandle cocycle quiver} of an oriented link diagram $D$, which is denoted by $SQ^{S,\theta}_{X}(D,a)$, is the pair $(SQ^{S}_{X}(D,a),\rho)$ of the shadow quandle coloring quiver $SQ^{S}_{X}(D,a)$ and a map $\rho:{\rm SCol}_{X}(D,a)\to A$ defined by $\rho(c_{\ast})=\Phi_{\theta}(D,c_{\ast})$.  
\end{defi}

\begin{defi}
Let $SQ^{S,\theta}_{X}(D,a)=(SQ^{S}_{X}(D,a),\rho)$ be the shadow quandle cocycle quiver of an oriented link diagram  $D$ 
and let $SQ^{S,\theta}_{X}(D^{\prime},a)=(SQ^{S}_{X}(D^{\prime},a),\rho^{\prime})$ be that of $D^{\prime}$. We say that $SQ^{S,\theta}_{X}(D,a)$ and $SQ^{S,\theta}_{X}(D^{\prime},a)$ are isomorphic if there is a bijection $\varphi:{\rm SCol}_{X}(D,a)\to{\rm SCol}_X(D^{\prime},a)$ which satisfies the following conditions.
\vspace{-0.5\baselineskip}
\begin{itemize}
	\setlength{\itemsep}{0pt}
	\setlength{\parskip}{0pt}
\item $\varphi$ determines  a quiver isomorphism from $SQ^{S}_{X}(D,a)$ to $SQ^{S}_{X}(D^{\prime},a)$.
\item For any $c_{\ast}\in{\rm SCol}_{X}(D,a)$, we have $\rho(c_{\ast})=\rho^{\prime}(\varphi(c_{\ast}))$.
\end{itemize}
\end{defi}

\begin{prop}
The shadow quandle cocycle quiver is a link invariant. Namely, for 
oriented link diagrams  $D$ and $D^{\prime}$   which are related by Reidemeister moves, the shadow quandle cocycle quivers $SQ^{S,\theta}_{X}(D,a)$ and $SQ^{S,\theta}_{X}(D^{\prime},a)$ are isomorphic.
\end{prop}

\begin{proof}
It is seen by the same argument with the proof that the quandle cocycle quiver is a link invariant given  in \cite{Cho2019cocycle}. 
\end{proof}

In \cite{Cho2019cocycle} a polynomial-valued link invariant, called the {\it quiver enhanced cocycle polynomial}, is introduced.  The following is the shadow version.  

\begin{defi}
Let $X$ be a finite quandle, $a$ an element of $X$, $S$ a subset of ${\rm End}(X)$, $A$ an abelian group and $\theta:X^3\to A$ a $3$-cocycle.  The {\it quiver enhanced shadow cocycle polynomial} of an oriented link diagram $D$ is the polynomial
\begin{eqnarray*}
\Phi^{S,\theta}_{X}(D,a)=\sum_{ (v, w) \, : \, {\rm an\ edge\ of\ }SQ^{S,\theta}_X(D,a)} s^{\rho(v)}t^{\rho(w)}
\end{eqnarray*}
where the summation is taken over all edges $(v, w)$ of $SQ^{S,\theta}_X(D,a)$.

Obviously, this polynomial is a link invariant.

\end{defi}
\begin{exam}
The knots $4_1$ and $5_1$ are not distinguished by the quandle coloring quivers using 
the dihedral quandle $R_5$ of order $5$ and any subset $S$ of End$(R_5)$, since $|{\rm Col}_{R_5}(4_1)|=|{\rm Col}_{R_5}(5_1)|=25$ (Theorem~\ref{p-dihed}). 

Moreover, they are not distinguished by the quandle cocycle quivers using $R_5$, any  subset $S$ of End$(R_5)$ and any $2$-cocycle $\theta: X^2 \to A$, since any $2$-cocycle of $R_5$ is a coboundary (in fact $H^2(R_5; A) = 0$, \cite{CJKS2001}) and the weight given to each vertex of the quiver is $0 \in A$.  Thus the quandle cocycle quiver is the quandle coloring quiver whose vertices are labeled with $0$. 

On the other hand, these knots are distinguished by shadow quandle cocycle quivers.  
The shadow quandle cocycle quivers $SQ^{S,\theta_p}_{R_p}(4_1,0)$ and $SQ^{S,\theta_p}_{R_p}(5_1,0)$ using $R_5$, Mochizuki's $3$-cocycle $\theta_5$ and $S=\{f\}$ where $f:R_5\to R_5$ is defined by $f(x)=x+2$, are as in Figures~\ref{coloring41} and \ref{coloring51}.  

 \begin{figure}[H]
 	\begin{minipage}{0.5\hsize}
  	\begin{center}
  	\includegraphics[height=4.375cm,width=6cm]{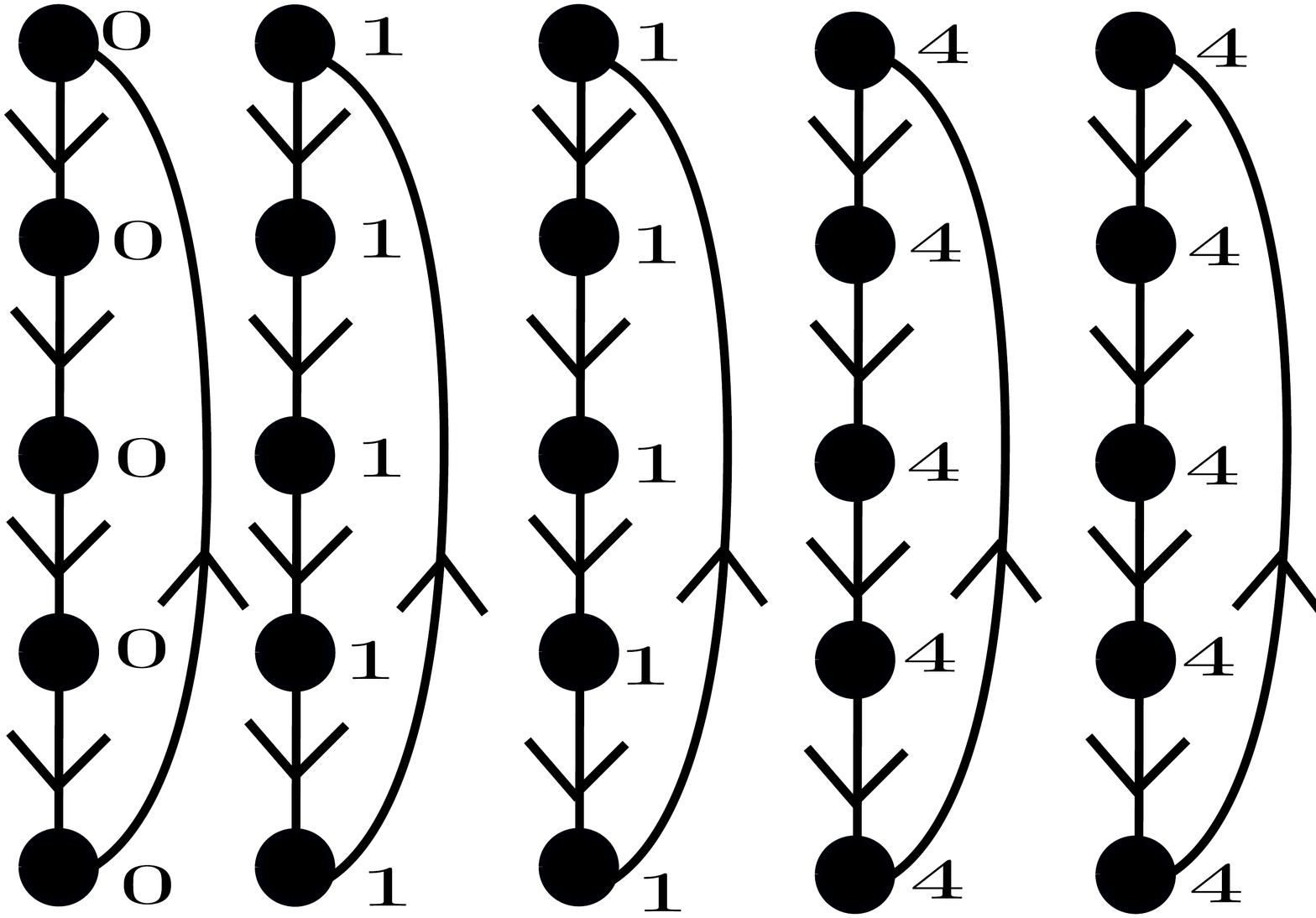}
  	\caption{$SQ^{S,\theta_p}_{R_p}(4_1,0)$}
  	\label{coloring41}
  	\end{center}
  	\end{minipage}
  	\begin{minipage}{0.5\hsize}
  	\begin{center}
  	\includegraphics[height=4.375cm,width=6cm]{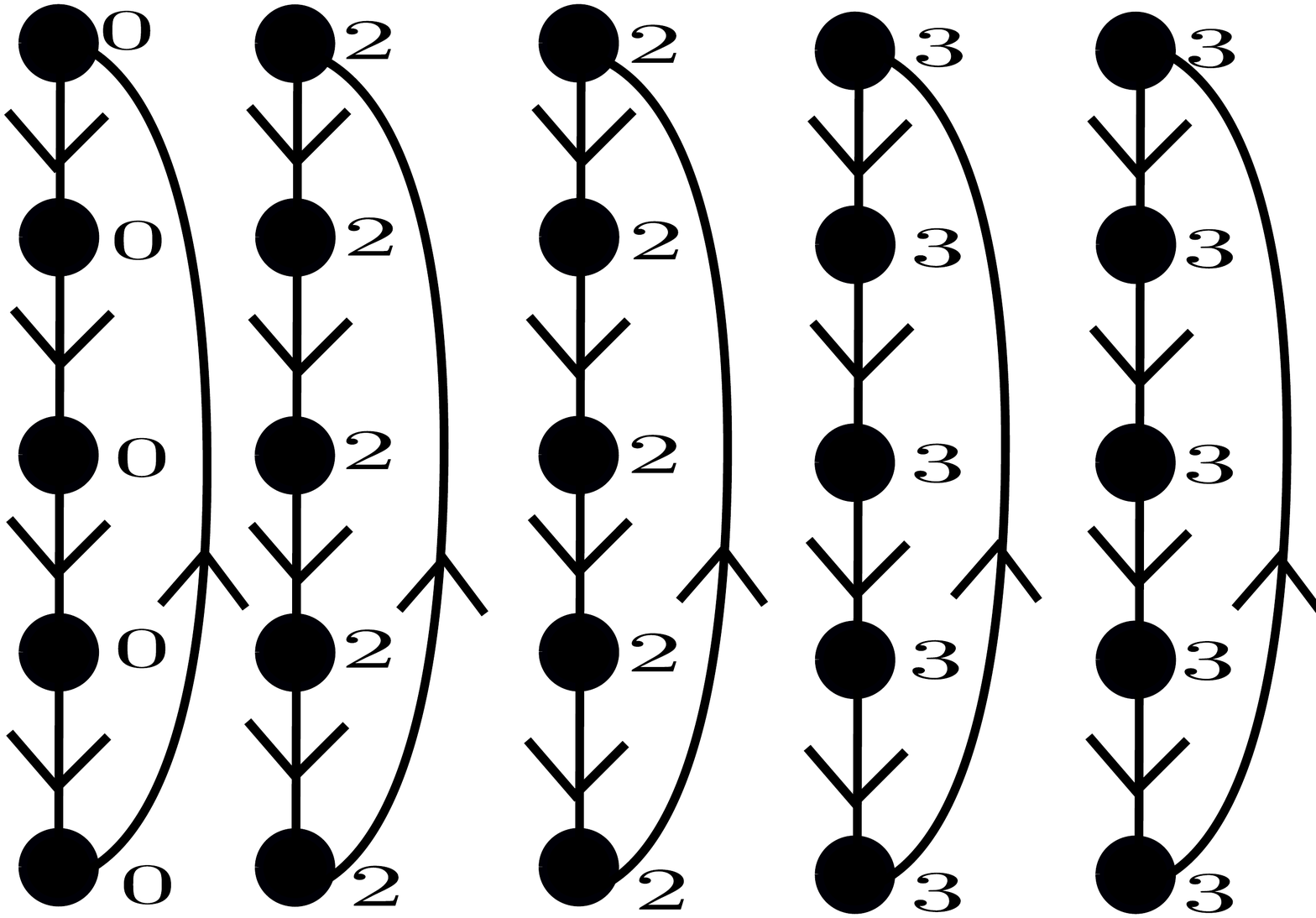}
  	\caption{$SQ^{S,\theta_p}_{R_p}(5_1,0)$}
  	\label{coloring51}
  	\end{center}
  	\end{minipage}
  \end{figure}

The quiver enhanced shadow cocycle polynomials are given as follows.
\begin{eqnarray*}
\Phi^{S,\theta_5}_{R_5}(4_1,0)&=&5+10st+10s^4t^4,\\
\Phi^{S,\theta_5}_{R_5}(5_1,0)&=&5+10s^2t^2+10s^3t^3.  
\end{eqnarray*} 
\end{exam}

\section{The shadow quandle cocycle quivers using  Mochizuki's $3$-cocycle}
\label{shadow quiver-Mochizuki}

In this section, we discuss the shadow quandle cocycle quivers using Mochizuki's $3$-cocycle. By Theorem \ref{p-dihed}, two different links which have the same $R_p$-coloring number can not be distinguished by the quandle coloring quivers using $R_p$.

 In the case of the shadow quandle cocycle quiver, we have a similar result. 
 
\begin{theo}
\label{cocycle-quiver}
Let $D$ and $D^{\prime}$ be oriented link diagrams, $p$  be an odd prime and $\theta_{p}$ be   Mochizuki's $3$-cocycle. For any $a\in R_{p}$ and $S\subset{\rm End}(R_p)$, 
the shadow quandle cocycle invariants $SQ^{S,\theta_p}_{R_p}(D,a)$ and $SQ^{S,\theta_p}_{R_p}(D^{\prime},a)$ are isomorphic if and only if 
$\Phi_{\theta_p}(D)=\Phi_{\theta_p}(D^{\prime})$. 
\end{theo}

Note that the set of shadow $R_p$-colorings of an oriented link diagram $D$  is a vector space over $\Z_p$. It is known that $\Phi_{\theta_p}(D)$ is characterized by the following lemmas.

\begin{lemm}[\cite{Satoh}]
\label{arc-cocycle}
Let $c_{\ast}$ and $c^{\prime}_{\ast}$ be shadow $R_p$-colorings of an oriented link diagram $D$ such that $c_{\ast}|_{{\rm Arc}(D)}=c^{\prime}_{\ast}|_{{\rm Arc}(D)}$. Then it holds that $\Phi_{\theta_p}(D,c_{\ast})=\Phi_{\theta_p}(D,c^{\prime}_{\ast})$.
\end{lemm}

\begin{lemm}[\cite{Satoh}]
\label{scalar-cocycle}
Let $c_{\ast}$ be a shadow $R_p$-coloring of an oriented link diagram $D$. Then it holds that $\Phi_{\theta_p}(D,ac_{\ast})=a^2\Phi_{\theta_p}(D,c_{\ast})$ for any $a\in\Z_p$.
\end{lemm}

\begin{lemm}[\cite{Satoh}]
\label{triv-cocycle}
Let $c_{\ast}$ and $c_{\ast,0}$ be shadow $R_p$-colorings of an oriented link diagram $D$ such that $c_{\ast,0}|_{{\rm Arc}(D)}$ is a trivial coloring. Then it holds that $\Phi_{\theta_p}(D,c_{\ast})=\Phi_{\theta_p}(D,c_{\ast}+c_{\ast,0})$.
\end{lemm}

\begin{proof}[Proof of Theorem \ref{cocycle-quiver}]
The only if part is trivial. We show the if part. 
By Lemma \ref{arc-cocycle}, we have $\{\Phi_{\theta_p}(D,c_{\ast})\mid c_{\ast}\in{\rm SCol}_{R_p}(D,a)\}=\{\Phi_{\theta_p}(D,c_{\ast})\mid c_{\ast}\in{\rm SCol}_{R_p}(D,b)\}$ for any $a,b\in R_p$. Therefore, $\Phi_{\theta_p}(D)=p\{\Phi_{\theta_p}(D,c_{\ast})\mid c_{\ast}\in{\rm SCol}_{R_p}(D,a)\}$ (cf. \cite{Satoh}).

We construct a bijection $\Psi$: SCol$_{R_p}(D,a)\to{\rm SCol}_{R_p}(D^{\prime},a)$ as follow:

Let $\varphi:{\rm Col}_{R_p}(D)\to{\rm SCol}_{R_p}(D,a)$ and $\varphi^{\prime}:{\rm Col}_{R_p}(D^{\prime})\to{\rm SCol}_{R_p}(D^{\prime},a)$ be the natural bijections as in the proof of Proposition \ref{shadow-quiver}.

Firstly, put $V_0=\{c_{\ast}\in{\rm SCol}_{X}(D,a)\mid c_{\ast}|_{{\rm Arc}(D)}\ {\rm is \ a\ trivial\ coloring}\}$ and $V^{\prime}_0=\{c^{\prime}_{\ast}\in{\rm SCol}_{X}(D^{\prime},a)\mid c^{\prime}_{\ast}|_{{\rm Arc}(D^{\prime})} \ {\rm is\ a\ trivial\ coloring}\}$. We define a map $\Psi_0:V_0\to V^{\prime}_0$ by $\Psi_0(c_{\ast})(x^{\prime})=c_{\ast}(x)$ for any $x\in{\rm Arc}(D)$ and $x^{\prime}\in{\rm Arc}(D^{\prime})$.
 
Next, take an element $c_{\ast,1}\in{\rm SCol}_{R_p}(D,a)\backslash V_0$ and fix it. By assumption, there is a shadow $R_p$-coloring $c^{\prime}_{\ast,1}\in{\rm SCol}_{R_p}(D^{\prime},a)\backslash V^{\prime}_0$ such that $\Phi_{\theta_p}(D, c_{\ast,1})=\Phi_{\theta_p}(D^{\prime},c^{\prime}_{\ast,1})$. 
Let $c_1$ and $c^{\prime}_1$ be the shadow colorings of $D$ and $D^{\prime}$ respectively such that 
for any 
$x\in{\rm Arc}(D)$ and $x^{\prime}\in{\rm Arc}(D^{\prime})$, $c_1(x) = c^{\prime}_1(x^{\prime}) =1 $ and $c_1(r_\infty) = c^{\prime}_1(r_\infty) =a$.

Put $V_1=\{\varphi(f\circ \varphi^{-1}(c_{\ast,1}))\mid f\in {\rm Aut}(R_p)\}$ and $V^{\prime}_1=\{\varphi^{\prime}(f\circ (\varphi^{\prime})^{-1}(c^{\prime}_{\ast,1}))\mid f\in {\rm Aut}(R_p)\}$. We define a map $\Psi_1:V_0\cup V_1\to V^{\prime}_0\cup V^{\prime}_{1}$ by $\Psi_1|_{V_0}=\Psi_0$ and $\Psi_{1}(\varphi(f\circ \varphi^{-1}(c_{\ast,1})))=\varphi^{\prime}(f\circ (\varphi^{\prime})^{-1}(c^{\prime}_{\ast,1}))$ for any $f\in{\rm Aut}(R_p)$. Since $f$ is an automorphism and $c^{\prime}_{\ast,1}$ is a nontirivial coloring, $\Psi_1$ is a bijection. Let $f$ be a quandle automorphism of $R_p$. By definition of $\varphi$, we have 
\begin{eqnarray*}
\varphi(f\circ \varphi^{-1}(c_{\ast,1}))|_{{\rm Arc}(D)}&=&f\circ c_{\ast,1}|_{{\rm Arc}(D)}\\
&=&(a_{f}c_{\ast,1}+b_fc_1)|_{{\rm Arc}(D)}
\end{eqnarray*}
 where $a_f \in\Z^{\times}_p$ and $b_f \in \Z_p $ are determined from $f$ by Lemma~\ref{end-dihed} (Remark \ref{endo-auto}) with $f(x)=a_fx+b_f$ . By Lemmas \ref{arc-cocycle}, \ref{scalar-cocycle} and \ref{triv-cocycle} , it holds that
 \begin{eqnarray*}
 \Phi_{\theta_p}(D,\varphi(f\circ \varphi^{-1}(c_{\ast,1})))&=&\Phi_{\theta_p}(D,a_fc_{\ast,1}+b_fc_{1})\\
 &=&(a_f)^2\Phi_{\theta_p}(D,c_{\ast,1}).
 \end{eqnarray*}
  Similarly, we have $\Phi_{\theta_p}(D^{\prime},\varphi^{\prime}(f\circ (\varphi^{\prime})^{-1}(c^{\prime}_{\ast,1})))=(a_f)^2\Phi_{\theta_p}(D^{\prime},c^{\prime}_{\ast,1})$ . By assumption, we see that
\begin{eqnarray*}
\Phi_{\theta_p}(D,\varphi(f\circ \varphi^{-1}(c_{\ast,1})))&=&(a_f)^2\Phi_{\theta_p}(D,c_{\ast,1})\\
&=&(a_f)^2\Phi_{\theta_p}(D^{\prime},c^{\prime}_{\ast,1})\\
&=&\Phi_{\theta_p}(D^{\prime},\varphi^{\prime}(f\circ (\varphi^{\prime})^{-1}(c^{\prime}_{\ast,1}))).
\end{eqnarray*}
 Then it holds that $\{\Phi_{\theta_p}(D,c_{\ast})\mid c_{\ast}\in V_0\cup V_1\}=\{\Phi_{\theta_p}(D^{\prime},c^{\prime}_{\ast})\mid c^{\prime}_{\ast}\in V^{\prime}_0\cup V^{\prime}_1\}$.

Suppose that ${\rm SCol}_{R_p}(D,a)\backslash (V_0\cup V_1) \neq \emptyset$.  
Take an element  $c_{\ast,2}$ of ${\rm SCol}_{R_p}(D,a)\backslash (V_0\cup V_1)$ and fix it.  There is a shadow $R_p$-coloring $c^{\prime}_{\ast,2}\in{\rm SCol}_{R_p}(D,a)\backslash (V^{\prime}_0\cup V^{\prime}_1)$ such that
$\Phi_{\theta_p}(D, c_{\ast,2})=\Phi_{\theta_p}(D^{\prime},c^{\prime}_{\ast,2})$. 
    Put $V_2=\{\varphi(f\circ \varphi^{-1}(c_{\ast,2}))\mid f\in {\rm Aut}(R_p)\}$ and $V^{\prime}_2=\{\varphi^{\prime}(f\circ (\varphi^{\prime})^{-1}(c^{\prime}_{\ast,2}))\mid f\in {\rm Aut}(R_p)\}$. Then, we define a bijection $\Psi_2:V_0\cup V_1\cup V_2\to V^{\prime}_0\cup V^{\prime}_{1}\cup V^{\prime}_2$ by $\Psi_1|_{V_0\cup V_1}=\Psi_1$ and $\Psi_{2}(\varphi(f\circ \varphi^{-1}(c_{\ast,2})))=\varphi^{\prime}(f\circ (\varphi^{\prime})^{-1}(c^{\prime}_{\ast,2}))$ for any automorphism $f\in{\rm Aut}(R_p)$.
Repeating this proceedure until the domain of $\Psi_n$ is ${\rm SCol}_{R_p}(D,a)$, we have a bijection $\Psi=\Psi_n:{\rm SCol}_{R_p}(D,a)\to{\rm SCol}_{R_p}(D^{\prime},a)$ such that $\Phi_{\theta_p}(D,c_{\ast})=\Phi_{\theta_p}(D^{\prime},\Psi(c_{\ast}))$ for any $c_{\ast}\in{\rm SCol}_{R_p}(D,a)$. By Remark \ref{endo-auto}, $\Psi$ is a quiver isomorphism between $SQ^{S}_{R_p}(D,a)$ and $SQ^{S}_{R_p}(D^{\prime},a)$ for any subset $S\subset{\rm End}(R_p)$.
\end{proof}

\section*{Acknowledgements} 

The author would like to thank Seiichi Kamada and Hirotaka Akiyoshi for helpful advice and discussions on this research.

\bibliographystyle{plain}
\bibliography{reference}

\end{document}